\documentclass[11pt]{amsart}
\usepackage{amsmath}
\usepackage{amssymb}
\usepackage{amsfonts}
\usepackage{latexsym}
\usepackage{color}
\usepackage{hyperref}
\usepackage{geometry}
\newcommand{\be}{\begin{equation}}
\newcommand{\en}{\end{equation}}
\newcommand{\bea}{\begin{eqnarray}}
\newcommand{\ena}{\end{eqnarray}}
\newcommand{\beano}{\begin{eqnarray*}}
\newcommand{\enano}{\end{eqnarray*}}
\newcommand{\bee}{\begin{enumerate}}
\newcommand{\ene}{\end{enumerate}}

\newcommand{\mc}{\mathcal}

\newcommand{\R}{\mathbb{R}}

\newcommand{\h}{\mathcal{H}}

\newcommand{\F}{{\mathcal F}}

\newcommand{\B}{{\mc B}}

\newcommand{\D}{{\mc D}}

\newtheorem{defn}{Definition}[section]
\newtheorem{thm}[defn]{Theorem}
\newtheorem{prop}[defn]{Proposition}
\newtheorem{lemma}[defn]{Lemma}
\newtheorem{cor}[defn]{Corollary}
\newtheorem{example}[defn]{Example}
\newtheorem{rem}[defn]{Remark}

\def\x{\relax\ifmmode {\mbox{*}}\else*\fi}
\newcommand{\beex}{\begin{example}$\!\!${\bf }$\;$\rm }
\newcommand{\enex}{ \end{example}}
\newcommand{\berem}{\begin{rem}$\!\!${\bf }$\;$\rm }
\newcommand{\enrem}{ \end{rem}}
\newcommand{\bedefi}{\begin{defn}$\!\!${\bf }$\;$\rm }
\newcommand{\findefi}{\end{defn}}

\newcommand{\ip}[2]{\left\langle {#1}\left|{#2}\right.\right\rangle}

\def\H{{\mathcal H}}
\def\K{{\mathcal K}}
\def\J{{\mathcal J}}

\begin{document}
\title[Continuous frames]
{Continuous frames for unbounded operators}

\author{Giorgia Bellomonte}

\address{Dipartimento di Matematica e Informatica,
Universit\`a degli Studi di Palermo, I-90123 Palermo, Italy}
\email{giorgia.bellomonte@unipa.it}

\date{\today}

\begin{abstract} Few years ago G\u{a}vru\c{t}a gave the notions of $K$-frame and  atomic system for a linear bounded operator $K$ in a Hilbert space $\H$ in order to decompose $\mathcal{R}(K)$, the range of $K$, with a frame-like expansion. These notions are here generalized to the case of a densely defined and possibly unbounded operator on a Hilbert space $A$ in a continuous setting, thus extending what have been done in a previous paper in a discrete framework. 
\end{abstract}

\maketitle
\section{Introduction}
The notion of discrete frame  was introduced by  Duffin and Schaefer in 1952 \cite{DuffinSchaeffer} even though it raised on the mathematical  and physical  scene in 1986 with the paper of  I. Daubechies, A. Grossmann, Y. Meyer because of their use in wavelet analysis. In the early '90s G. Kaiser \cite{Kaiser} and (independently) S.T. Ali, J.P. Antoine and J.P. Gazeau \cite{AAG} extended this notion to the continuous case. Over the years many extensions of frames have been introduced and studied. Most of them have been considered in the discrete case because of their wide use in applications e.g. in signal processing \cite{DuffinSchaeffer}. Frames have been studied for the whole Hilbert space or for a closed subspace until 2012, when L. G\u{a}vru\c{t}a \cite{gavruta}  gave  the notions of $K$-frame and of atomic system for a bounded operator $K$ everywhere defined on $\H$, thus generalizing the notion of frame and that of  atomic system for a subspace due to H.G. Feichtinger and T. Werther \cite{FW}. $K$-frames allow to write each element of $\mathcal{R}(K)$, the range
 of $K$, which is not a closed subspace in general, as a combination of the elements of the $K$-frame, which do
 not necessarily belong to $\mathcal{R}(K)$ with $K\in \B(\H)$.  $K$-frames have been generalized in \cite{AR} and \cite{HH} where the notion of $K$-g-frames was investigated and have been further generalized in 2018 to the continuous case in \cite{AROR}.   
 \\

 Let $\H$ be a Hilbert space with inner product $\ip{\cdot}{\cdot}$ and norm $\|\cdot\|$, $(X,\mu)$ a measure space where $\mu$ is a positive measure and  $A$  a densely defined operator on $\H$. Let $\psi:x\in X\to\psi_x\in\H$ be a  Bessel function, i.e. $\psi$ be such that for all $f\in\H$, the map $x \to \ip{f}{\psi_x}$ is a measurable function on $X$ and there exists a constant $\beta>0$ such that
 $ \int_X
 |\ip{f}{\psi_x}|^2 d\mu(x) \leq \beta \|f\|^2$,\ $\forall f\in\H$. Assume that for $f\in \D(A)$ (the domain of $A$) we have the decomposition

 $$\ip{Af}{u}=\int_X  a_f(x) \ip{\psi_x }{u}d\mu(x), \qquad \forall u\in \D(A^*).$$
 for some 	$a_f\in L^2(X,\mu)$. If $A$ is unbounded, the function $a_f$ can not depend continuously on $f$, differently to what occurs when $A$ is bounded. In  order to decompose the range of a densely defined unbounded operator $A$ as a combination of vectors in $\H$, we need somewhat which takes on its unboundedness.
 In literature there are some generalizations to the continuous case of the notion of $K$-frame (as, e.g., c-$K$-g-frames in \cite{AROR}), however, as far as the author knows, the case of an unbounded operator $K$ in $\H$ has been little considered. 

 In \cite{BC} this problem has been addressed in the discrete case. In the present paper both the approaches introduced in \cite{BC} are  extended  to the {\it continuous setting}.  One of the approaches involves a  Bessel function   $\psi$  and the coefficient function $a_f$  depends continuously on $f\in \D(A)$ only in the graph topology of $A$,  (which is stronger than the norm of $\H$); the other one involves a non-Bessel function $\psi$ but the coefficient function $a_f$  depends continuously on $f\in \D(A)$. In the latter approach,  the notions of continuous weak $A$-frame and  continuous weak atomic system for an unbounded operator $A$ are introduced and studied.
 
 { If   $\theta:X\to\H$ is a continuous frame for $\H$ then of course $$\ip{Af}{h}=\int_X \ip{Af}{\zeta_x}\ip{\theta_x}{h}d\mu(x),\qquad\forall f\in \D(A), h\in\H$$ where $\zeta:X\to\H$ is a dual frame of $\theta$. In contrast, if $\psi$ is a continuous weak $A$-frame, then there exists a Bessel function $\phi:X\to\H$ such that 	
 	$$
 	\ip{A h}{u}=\int_X  \langle h | \phi_x\rangle \ip{\psi_x }{u}d\mu(x),\qquad \forall h\in \D(A), u\in \D(A^*)
 	$$
 	and the action of the operator $A$ does not appear in the weak decomposition of the range of $A$.
 	Still, 	continuous weak $A$-frames  clearly call to mind continuous multipliers which are the object of interest of a recent literature even though unbounded multipliers, as far the authors knows, have been little looked over. For example, some initial steps toward this direction has been done, in the discrete case, in  	\cite{BagBell,BagBell2,Bag_Riesz,Bag_sesq,inoue} where some unbounded multipliers have been defined.  Therefore this paper can spure investigation in the direction of unbounded multipliers in the continuous case. \\
 
The paper is organized as follows. In Sect.  \ref{sec: preliminaries} we recall some well known definitions and introduce the generalized frame operator $S_\Psi$ which is the operator associated to a sesquilinear form defined by means of a function $\psi:x\in X\to\psi_x\in\H$. In Sect. \ref{sec: cont weak} we introduce, prove the existence (under opportune hypotheses) and study  the notions of continuous weak $A$-frame and continuous weak atomic system for a densely defined operator $A$ in a Hilbert space $\H$. To go into more detail, after having introduced and studied the notion of continuous weak $A$-frame, Subsection \ref{subsec: frame related operators} is devoted to the study of frame-related operators as the analysis, synthesis and (generalized) frame operators of a continuous weak $A$-frame.  In Subsection \ref{subsec: cont atomic syst}  the notion of continuous weak atomic system for an unbounded operator $A$ in Hilbert space $\H$ is given. 
Under some hypotheses, this notion is equivalent to that of continuous weak $A$-frame. Moreover, given a suitable function $\psi:x\in X\to\psi_x\in\H$, for every bounded operator $M\in\B(\h, L^2(X,\mu))$, an operator $A_M$ can be constructed in order   $\psi$  to be a continuous weak atomic system for $A_M$.  
Section \ref{sec: cont K-frames} is devoted to the second approach to the problem of decomposing the range of an unbounded operator in Hilbert space: we consider a bounded operator $K$ from a Hilbert space $\J$ into another one $\H$ and give some  results about both  continuous $K$-frames and continuous atomic systems for $K$ and about their frame-related operators, then in  Subsection \ref{subsec: cont atomic sys for unbounded op}, we use them to study the case of an unbounded closed and densely defined operator $A:\D(A)\to\H$ viewing it as a bounded one $A:\H_A\to\H$, where $\H_A$ is the Hilbert space obtained by giving $\D(A)$ the graph norm.

\section{Definitions and preliminary results}\label{sec: preliminaries}
Throughout the paper we will denote by $\H$ a complex Hilbert space with inner product $\ip{\cdot}{\cdot}$ (linear in the first entry) and induced norm $\|\cdot\|$, by $(X,\mu)$ a $\sigma$-finite measure space (i.e. $X$ can be covered with at most countably many measurable,  possibly disjoint, sets $\{X_n\}_{n\in\mathbb{N}}$ of finite measure), 
 by $\B(\H)$  the Banach space of bounded linear operators from $\H$ into $\H$.
For brevity we will indicate by
$L^2(X,\mu)$  the class of all $\mu$-measurable functions
$f : X\to\mathbb{C}$ such that
$$\|f\|_2^2=\int_X|f(x)|^2d\mu(x)<\infty,$$ by identifying functions which differ only on a $\mu$-null subset of $X$.
\\

Let us briefly recall the notion of continuous frame (see e.g. \cite[Definition 4.1]{Kaiser}, \cite[Definition 2.1]{AAG},  \cite[Definition 5.6.1]{ole})\begin{defn}
A continuous frame for $\H$ is a function $\psi:x\in X\to\psi_x\in\H$ for which
	\begin{itemize}
		\item[i)] for all $f\in\H$, the map $x \to \ip{f}{\psi_x}$ is a measurable function on $X$ (i.e. the function $\psi$ is  weakly measurable),
		\item[ii)] there exist constants $\alpha,\beta>0$ such that
		\begin{equation}\label{defn cont frame for H}\alpha \|f\|^2 \leq \int_X
		|\ip{f}{\psi_x}|^2 d\mu(x) \leq \beta \|f\|^2, \qquad\forall f\in\H.
		\end{equation}	
	\end{itemize}
	The function $\psi$ is called \emph{a Bessel function}  if at least the upper condition in \eqref{defn cont frame for H} holds. If $\alpha=\beta=1$ then the function $\psi$ is called a  Parseval function.
\end{defn}

The main feature of a frame, hence of a continuous
	frame too, is the possibility of writing each vector of a  Hilbert space as a sum of a infinite linear  combination of vectors in the space getting rid of rigidness of orthonormality of the vectors of a basis and of the uniqueness 	of the decomposition, but still  maintaining  
	numerical stability of the reconstruction and fast convergence. By a continuous frame it is possible to represent every element of the Hilbert space by a  reconstruction
	formula: if $\psi:x\in X\to\psi_x\in\H$ is a continuous frame for the
	Hilbert space $\H$, then any $f\in\H$ can be expressed as
$$f =\int_X
	\ip{f}{\phi_x}\psi_x d\mu(x),$$
	where  $\phi:x\in X\to\phi_x\in\H$ is  a function called  dual of $\psi$ and the integrals have to be understood
	in the weak sense, as usual.

\subsection{Frame-related operators and sesquilinear forms}\label{subsec: frame related op and sesq form}In this section we recall the definitions of the main operators linked to a $\psi:x\in X\to\psi_x\in\H$ and prove some results about them.  We want to drive the attention of the reader on the fact that, in contrast with the discrete case where some results involve strong convergence \cite{BC}, in the continuous case we can prove our results just in  weak sense.\\

In the sequel we will  briefly indicate the range $\{\psi_x\}_{x\in X}$ of a function $\psi:x\in X\to\psi_x\in\H$  by $\{\psi_x\}$. \\

Consider the function $\psi:x\in X\to\psi_x\in\H$ and the set $$
	\D(C_\psi)=\left \{f\in\H: \int_X |\ip{f}{\psi_x}|^2d\mu(x)<\infty\right \}
.$$ The  operator $C_\psi:h\in\D(C_\psi)\subset\H\to \ip{h}{\psi_x}\in L^2(X,\mu)$  (strongly) defined, for every $h\in\D(C_\psi)$ and for every $x\in X$, by \begin{equation}\label{defn: C operator}(C_\psi h)(x)=\ip{h}{\psi_x}\end{equation} is called {\it the analysis operator of the function} $\psi$ (borrowing the terminology from frame theory).

\berem In general the domain of $C_\psi $ is not dense, hence $C_\psi ^*$ is not well-defined. An example of function whose analysis operator is densely defined can be found in Example \ref{S subset T_Omega}, where $\D(C_\psi )=\D(\Psi)$. Moreover, a sufficient condition for $\D(C_\psi )$ to be dense in $\H$ is that $\psi_x\in\D(C_\psi ) $ for every $x\in X$ (see Lemma 2.3. \cite{AB}). \enrem

The next result will be often needed in Section \ref{sec: cont weak}. 	In contrast with \cite[Lemma 2.1]{AB} we do not suppose that $\{\psi_x\}$ is total.
\begin{prop}\label{prop: closedness C} Let $(X,\mu)$ be a measure space and $\psi:x\in X\to\psi_x\in\H$. The analysis  operator $C_\psi $  is  closed.\end{prop}\begin{proof} Consider any sequence $\{h_n\}\subset\D(C_\psi )$ such that $h_n\stackrel{\|\cdot\|}{\to} h$  in $\H$ and $C_\psi h_n\stackrel{\|\cdot\|_2}{\to} c$ in $L^2(X,\mu)$ as $n\to\infty$; we shall prove that $h\in\D(C_\psi )$ and $C_\psi h=c$.  For every $x\in X$  the functionals $f\in \D(C_\psi )\to\ip{f}{\psi_x}\in\mathbb{C}$ are continuous hence, as $n\to\infty$, $$g_n(x)=\ip{h_n}{\psi_x}\to g(x)=\ip{h}{\psi_x},$$ in $L^2(X,\mu)$. By \cite[Theorem 3.12]{RudinRCA} there exists a subsequence $g^{(k)}_{n_l}=\ip{h_{n_l}}{\psi_x}$ such that $l\to\infty$ $g^{(k)}_{n_l}=\ip{h_{n_l}}{\psi_x}\to g=\ip{h}{\psi_x}$ a.e. on $X$.   Then for every $x\in X$ we have that for $n\to \infty$  $$(C_\psi h_n)(x)=\ip{h_n}{\psi_x}\to c(x),$$ hence $g(x)=\ip{h}{\psi_x}=c(x)$, i.e. $h\in\D(C_\psi )$ and $(C_\psi h)(x)=c(x)$, therefore $C_\psi $ is closed.
\end{proof}

If $C_\psi $ is densely defined, let us calculate its adjoint operator: let $a\in\D(C_\psi^*)$ with $\D(C_\psi^*)=\{a\in L^2(X,\mu): \exists g\in\H  \mbox{ such that} \ip{C_\psi h}{a}_2 = \ip{h}{g},\,\forall h\in\D(C_\psi )\}$

$$\ip{C_\psi^* a}{h}=\ip{a}{C_\psi h}_2=\int_Xa(x)\ip{\psi_x}{h}d\mu(x),\quad h\in\D(C_\psi )$$hence

$C_\psi^*:\D(C_\psi^*)\subset L^2(X,\mu)\to \H$ is weakly defined by:
\begin{equation*}
\ip{C_\psi^*a}{h}=\int_{X}\ip{a(x)\psi_x}{h}d\mu(x),\qquad a\in \D(C_\psi^*),\,h\in\D(C_\psi )
\end{equation*} and is called the {\it synthesis operator of the function} $\psi$ where
$$\D(C_\psi ^*):=\left \{a\in L^2(X,\mu):\int_X\ip{a(x)\psi_x}{h}d\mu(x) \mbox{ exists } \forall h\in \D(C_\psi )\right \}.$$

\berem Thus, if  $C_\psi $ is densely defined, then the synthesis operator $C_\psi^*$ is a densely defined  closed operator. \enrem

\begin{prop}[\cite{FO}]\label{prop: syntesis and analysis operator}
	The function $\psi:x\in X\to\psi_x\in\H$ is Bessel with bound $\beta>0$ if and only if the synthesis operator $C_\psi^*$
	is linear and bounded on $L^2(X,\mu)$ with $\|C_\psi^*\|_{L^2,\H}\leq\sqrt{\beta}$. Moreover, the analysis operator  $C_\psi $ is linear and bounded on  $\H$ with $\|C_\psi \|_{\H, L^2}\leq\sqrt{\beta}$. More precisely $\|C_\psi^*\|_{L^2,\H}=\|C_\psi \|_{\H, L^2}=\sup_{f\in \H,\|f\|=1}\left( \int_{X}\left|\ip{f}{\psi_x}\right|^2d\mu(x)\right) ^{1/2}\leq\sqrt{\beta}.$ 	\end{prop}

Extending to the continuous case \cite{corso}, consider the set
$$\D(\Psi)=\left \{f\in\H: \int_X |\ip{f}{\psi_x}|^2d\mu(x)<\infty\right \}=\D(C_\psi)$$
and the mapping
$\Psi: \D(\Psi) \times\D(\Psi)\to\mathbb{C}$ defined by \begin{equation}\label{defn: Omega}\Psi(f, g) :=\int_X\ip{f}{\psi_x}\ip{\psi_x}{g}
d\mu(x).\end{equation} $\Psi$
is clearly a nonnegative symmetric sesquilinear form which is  well defined for every $f,g\in\D(\Psi)$ because of the Cauchy-Schwarz inequality. It is unbounded in general.

Moreover, since $\D(\Psi)$ is the largest domain such that $\Psi$ is defined on $\D(\Psi)\times\D(\Psi)$, it results that \begin{equation}\label{eq: Omega closed}
\Psi(f,g)=\ip{C_\psi f}{C_\psi g}_2, \quad\forall f,g\in\D(C_\psi )=\D(\Psi)
\end{equation}
where $C_\psi $ is the analysis operator defined in \eqref{defn: C operator}.  Since $C_\psi $ is a closed operator,
$\Psi$ is a closed nonnegative symmetric sesquilinear form in $\H$ (see e.g. \cite[Example 1.13]{Kato}).

Let us assume that $\D(\Psi)$ is dense in $\H$, then by Kato's first representation
theorem \cite[Theorem VI.2.1]{Kato} there exists a positive  self-adjoint  operator $S_\Psi$ {\em associated to the sesquilinear form} $\Psi$  on \begin{eqnarray}\label{eq: Omega}
\D(S_\Psi) &=&\{f\in \D(\Psi): h\to\int_X\ip{f}{\psi_x} \ip{\psi_x}{h}d\mu(x)\\\nonumber && \mbox{ is bounded on }\D(\Psi) \mbox{ w.r. to }\|\cdot\|\}
\end{eqnarray}defined  by \begin{equation}\label{def: Frame opera}S_\Psi f:=h\end{equation}  with $h$ as in \eqref{eq: Omega} ($h$ is uniquely determined because of the density of $\D(\Psi)$). The operator $S_\Psi$ is the  greatest one whose domain is contained in $\D(\Psi)$ and
for which  the following  representation holds
$$\Psi(f, g) = \ip{S_\Psi f}{g},\qquad  f \in \D(S_\Psi),\, g\in\D(\Psi).$$
The set $\D(S_\Psi)$ is dense in $\D(\Psi)$ (see  \cite[p. 279]{Kato}).

Furthermore, by Kato's second representation theorem \cite[Theorem VI.2.23]{Kato}, $\D(\Psi)=\D(S_\Psi^{1/2})$
and \begin{equation*}\label{eq: omega and T}
\Psi(f, g) = \ip{S_\Psi^{1/2}f}{S_\Psi^{1/2}g},\qquad \forall  f, g \in \D(\Psi)
\end{equation*} and comparing with \eqref{eq: Omega closed}, we obtain $S_\Psi=C_\psi ^*C_\psi =|C_\psi |^2$ on $\D(S_\Psi)$.

\bedefi
The operator $S_\Psi:\D(S_\Psi)\subset\H\to\H$ defined by \eqref{def: Frame opera} will be said the {\em generalized frame operator} of the function $\psi:x\in X\to\psi_x\in\H$. \findefi

Given $\psi:x\in X\to\psi_x\in\H$, coherently with \cite{AB}, the operator $S_\psi:\D(S_\psi)\subset\H\to\H$
weakly defined by \begin{equation*}\label{eq: frame operator}
\ip{S_\psi f}{g}=\int_X \ip{f}{\psi_x}\ip{\psi_x}{g}d\mu(x),\qquad f\in\D(S_\psi ), g\in\H
\end{equation*}where $$\D(S_\psi )=\{f\in \H:\int_X \ip{f}{\psi_x}\psi_xd\mu(x)\,\, \mbox{converges weakly in }\H\}$$ is called the {\em frame operator} of  $\psi$. It is a positive operator on its domain 
and symmetric
indeed for every $f,g\in\D(S_\psi )$
\begin{eqnarray*}
	\ip{S_\psi f}{g}&=&\int_X \ip{f}{\psi_x}\ip{\psi_x}{g}d\mu(x)=\int_X\overline{\ip{\psi_x}{f}}\overline{\ip{g}{\psi_x}}d\mu(x)\\
	&=& \overline{\int_X\ip{g}{\psi_x}\ip{\psi_x}{f}d\mu(x)}=\ip{f}{S_\psi g},
\end{eqnarray*} but non densely defined in general. 
If $\psi$ is a continuous frame for $\H$, then the frame operator $S_\psi $ is a bounded operator in $\H$, positive, invertible with bounded inverse (see e.g. \cite{AAG}).
\\

\berem The generalized frame operator $S_\Psi$ and the frame operator $S_\psi $  coincide on $\D(S_\psi )\subset\D(S_\Psi)$. If in particular $\psi$ is a continuous frame for $\H$, then $C_\psi , S_\psi $ are defined on the whole $\H$ and $C_\psi^*$ on the whole $L^2(X,\mu)$ (see also \cite{AB}) and $S_\Psi=C_\psi^*C_\psi =S_\psi $ on $\H$. However,   in general, they are not the same operator,  as the following example shows.\enrem

\beex\label{S subset T_Omega} Let $X$ be such that $\mu(X)=\infty$ and having a covering made up of a countable collection $\{X_n\}_{n\in\mathbb{N}}$ of disjoint measurable subspaces of $X$ each of measure $M>0$, $\H$ a separable Hilbert space and $\{e_n\}_{n\in\mathbb{N}}$ an orthonormal basis of  $\H$.  Let $\alpha>1,\beta>0$ and define $\psi:x\in X\to\psi_x\in\H$ with
$$\psi_x:=\begin{cases} \psi_{2n-1}= n^\beta e_n, & \mbox{if } x\in X_{2n-1}\\\psi_{2n}= (n+1)^\alpha(e_{n+1}-e_n), & \mbox{if } x\in X_{2n}.
\end{cases}
$$
Then
$$\D(\Psi)=\left\{f\in\H: M\left(\sum_{n=1}^\infty n^{2\beta}|\ip{f}{e_n}|^2+\sum_{n=1}^\infty (n+1)^{2\alpha}|\ip{f}{e_{n+1}-e_n}|^2 \right)<\infty\right \}$$ is dense. Indeed, consider the sequence $\{\psi_n\}_{n\in\mathbb{N}}\subset\H$, then for every $m\in\mathbb{N}$

\begin{eqnarray*}
	\int_X\left |\ip{\psi_m}{\psi_x}\right |^2d\mu(x)&=&\sum_{k=1}^{\infty}\int_{X_k}\left|\ip{\psi_m}{\psi_k}\right|^2 d\mu(x)<\infty
\end{eqnarray*}because only two, three, or six   addendi in the series are different from zero, depending on the value of $m$. Then  $span\{\psi_x\}=span\{\psi_n\}\subset\D(\Psi)$. On the other hand $(span\{\psi_n\})^\perp\subset\D(\Psi)$, hence $$\H=\overline{span}\{\psi_n\}\oplus(\overline{span}\{\psi_n\})^\perp\subset\overline{\D(\Psi)}$$ hence $\D(\Psi)$ is dense in $\H$. 
We shall prove that there exists a $f\in\D(S_\Psi)$ such that $f\notin\D(S_\psi )$.
Let $f\in\H$ be such that $\ip{f}{e_n}=\frac{1}{n^p}$ for every $ n\in\mathbb{N}$, for a fixed $p\in\mathbb{N}$. We want to calculate for which values of $\alpha$ and $\beta$ such an $f\in\H$ is in $\D(S_\Psi)\setminus\D(S_\psi )$. For $f\in\D(\Psi)$ it has to be
\begin{equation}\label{eq: sum of series}
M\sum_{n=1}^\infty \frac{n^{2\beta}}{n^{2p}}+M\sum_{n=1}^\infty (n+1)^{2\alpha}\frac{\left |n^p-(n+1)^p\right |^2}{n^{2p}(n+1)^{2p}}<\infty.
\end{equation}
For $p>\beta+\frac{1}{2}$ the first series in \eqref{eq: sum of series}  converges, the second has general term that behaves like $\frac{1}{n^{2(p-\alpha+1)}}$ hence
if $p>\alpha-\frac{1}{2}$ too, then the series converges. 
To be $f\in\D(S_\Psi)$  the functional $g\in\D(\Psi)\to\int_X\ip{f}{\psi_x}\ip{\psi_x}{g}d\mu(x)$ has to be bounded. Take any $g\in\D(\Psi)$, then $\int_X\ip{f}{\psi_x}\ip{\psi_x}{g}d\mu(x)=M\ip{\sum_{n=1}^\infty\ip{f}{\psi_n}\psi_n}{g}$. Let us consider the sequence of partial sums  of the series $\sum_{n=1}^\infty\ip{f}{\psi_n}\psi_n$: 
\begin{eqnarray*}\label{eq: series in ex}
\nonumber s_{2m-1}&=&\sum_{n=1}^m\ip{f}{\psi_{2n-1}}\psi_{2n-1}+\sum_{n=1}^{m-1}\ip{f}{\psi_{2n}}\psi_{2n}\\	
		&=&ae_1+\sum_{n=2}^{m-1}b_n(p)e_n+c_m(p)e_m
\end{eqnarray*}
and \begin{eqnarray*}\label{eq: series in ex even term}
	\nonumber s_{2m}&=&\sum_{n=1}^m\ip{f}{\psi_{2n-1}}\psi_{2n-1}+\sum_{n=1}^{m}\ip{f}{\psi_{2n}}\psi_{2n}\\	
	&=&ae_1+\sum_{n=2}^{m}b_n(p)e_n+d_{m+1}(p)e_{m+1}
\end{eqnarray*}
with $a=\left[1+2^{2\alpha}\left(1-\frac{1}{2^p}\right) \right]>0$, $$b_n(p)=\frac{n^{2\beta}}{n^p}+\frac{n^{2\alpha}[(n-1)^p-n^p]}{n^p(n-1)^p}-
\frac{(n+1)^{2\alpha}[n^p-(n+1)^p]}{n^p(n+1)^p}
=\frac{n^{2\beta}}{n^p}+b'_n(p)$$ and 
$$c_m(p)=\frac{m^{2\beta}}{m^p}+d_m(p),\qquad d_{m+1}(p)=
\frac{(m+1)^{2\alpha}[m^p-(m+1)^p]}{m^p(m+1)^p}$$ 
where $b'_n(p)=\frac{p(p-1)}{n^{p-2\alpha+2}}+o\left(\frac{1}{n^{p-2\alpha+2}}\right)$ and  $d_m(p)=\frac{-p}{n^{p-2\alpha+1}}+o(\frac{1}{n^{p-2\alpha+1}})$. For $p>2\beta+\frac{1}{2}$ and $p>2\alpha-\frac{3}{2}$ the sequence $\{b_n(p)\}$ belongs to $\ell^2$.  Moreover, for every $g\in\D(\Psi)$ we have that $n^\beta|\ip{e_n}{g}|\to 0$, hence, if also $p\geq 2\alpha-\beta-1$, then $|c_m(p)\ip{e_m}{g}|\leq (1+p)m^\beta|\ip{e_m}{g}|\to 0$ and $|d_m(p)\ip{e_m}{g}|\leq pm^\beta|\ip{e_m}{g}|\to 0$ as $m\to\infty$.  Hence, the series $\ip{\sum_{n=1}^\infty\ip{f}{\psi_n}\psi_n}{g}$ 
converges. 
Now we want to calculate values of $\alpha$ and $\beta$ in order $f\notin\D(S_\psi )$. A vector $h\in\D(S_\psi )$ if and only if for every $g\in\H$
  $$\left |\int_X \ip{h}{\psi_x}\ip{\psi_x}{g}d\mu(x)\right |=\left |M\ip{\sum_{k=1}^{\infty}\ip{h}{\psi_k}\psi_k}{g}\right |<\infty$$ i.e. if the series $\sum_{k=1}^{\infty}\ip{h}{\psi_k}\psi_k$ weakly converges in $\H$, however, if $h=f$ and $0<2\alpha-1-p<\beta$ the norm of $s_k$ goes to infinity as $k\to\infty$.

As an example, if $p=3$ it can be   $\alpha=\frac{17}{8}$ and $\beta=\frac{1}{3}$ or, as in \cite{RC2}, $p=2$, $\alpha=\frac{8}{5}$ and $\beta=\frac{1}{2}$.
\enex

\begin{prop}\label{prop: S subset C^*C}Let $(X,\mu)$ be a $\sigma$-finite measure space, $\psi:x\in X\to\psi_x\in\H$  and $\D(\Psi)$ be dense. Then   the frame operator $S_\psi $  is closable.
\end{prop}
\begin{proof}The sesquilinear form $\Psi$ is nonnegative closed and densely defined ($\D(C_\psi )=\D(\Psi)$), hence the generalized frame  operator $S_\Psi$ is self-adjoint. We conclude the proof by recalling that  $S_\psi \subset S_\Psi$. \end{proof}

	In the following sections we will use the next two lemmas.	

\begin{lemma}\cite{BR}\label{lem: pseudoinverse unb} Let $\H,\K$ be Hilbert spaces. Let $W:\D(W)\subset\K\to\H$ a closed, densely defined operator with closed range $\mathcal{R}(W)$. Then, there exists a unique  $W^\dag\in\B(\H,\K)$  such that $$\mathcal{N}(W^\dag)=\mathcal{R}(W)^\perp,\,\, \overline{\mathcal{R}(W^\dag)}=\mathcal{N}(W)^\perp,\,\, WW^\dag f=f, \qquad f\in\mathcal{R}(W).$$
\end{lemma}
The operator $W^\dag$ is called the {\em pseudo-inverse} of the operator $W$.\\

The following lemma is a partial variation of two  Douglas majorization theorems \cite[Theorem 1, Theorem 2]{doug}, see also \cite{BC}.
		\begin{lemma}  	\label{doug gen unb}
	Let $(\H,\|\cdot \|),(\H_1,\|\cdot \|_1)$ and $(\H_2,\|\cdot \|_2)$ be Hilbert spaces and $T_1:\D(T_1)\subseteq \H_1\to \H$, $T_2:\D(T_2)\subseteq \H\to \H_2$ densely defined operators. 
	Assume that
	$T_1$ is closed and
	$\D(T_1^*)=\D(T_2)$.\\ Consider the following statements\begin{itemize}
		\item[i)]  $\|T_1^* f\|_1\leq \lambda \|T_2f\|_2$ for all $f\in \D(T_1^*)$ and some $\lambda>0$,
		\item[ii)] there exists a bounded operator $U\in \B(\H_1,\H_2)$ such that $T_1=T_2^* U$.
	\end{itemize}Then $i)\Rightarrow ii)$.
	
	If, in addition, $T_2$ is a bounded operator on $\H$, then $i)\Leftrightarrow ii)$ and both are equivalent to
	\begin{itemize}
		\item[iii)] $\mathcal{R}(T_1)\subset\mathcal{R}(T_2^*)$.
	\end{itemize}
\end{lemma}

\section[Continuous weak $A$-frame]{Continuous weak $A$-frame and continuous atomic systems for unbounded operators}\label{sec: cont weak}
	
	In this section we introduce and study our extension to the continuous case of the notions of discrete weak $A$-frame and  discrete weak atomic system for a  densely defined operator $A$ on a Hilbert space, given in \cite{BC}.
	\bedefi \label{def_ cont weak A-Frames}
	Let $A$ be a  densely defined operator on $\H$. A {\it continuous weak $A$-frame} for $\H$ is a function $\psi:x\in X\to\psi_x\in\H$ such that for all $f\in\D(A^*)$, the map $x \to \ip{f}{\psi_x}$ is a measurable function on $X$ and
	\begin{equation*}
	\alpha \|A^* f\|^2\leq\int_X
	|\ip{f}{\psi_x }|^2d\mu(x)<\infty,
	\end{equation*}for every $f\in\D(A^*)$ and some $\alpha>0$.
	\findefi

	\berem If $X=\mathbb{N}$ and $\mu$ is a counting measure, a continuous weak $A$-frame clearly reduces to a discrete weak $A$-frame in the sense of \cite{BC}.\enrem

		\berem Let $(X,\mu)$ be a $\sigma$-finite measure space. If $A\in\B(\H)$, a continuous weak $A$-frame is  a continuous $A$-g-frame in the sense of \cite[Definition 2.1]{AROR} with $\Lambda_x=\ip{f}{\psi_x}$ for every $f\in\H$, with $x\in X$, since $C_\psi $  is a bounded operator in that case. \enrem

	\berem\label{ex: frame  to continuous weak A frame} Let $A$ be a  densely defined operator on $\H$ and   $\psi:x\in X\to\psi_x\in\D(A)\subset\H$ a continuous frame for $\H$. Then $A \psi$ is a  continuous weak $A$-frame for $\H$. Indeed, there exist constants $\alpha,\beta>0$ such that
	\begin{equation*}\alpha \|A^*f\|^2 \leq \int_X
	|\ip{A^*f}{\psi_x}|^2 d\mu(x) \leq \beta \|A^*f\|^2, \qquad\forall f\in\D(A^*).
	\end{equation*}	

	\enrem

	\beex
		Let $X = \mathbb{R}^2$ and let $\mu$ be the Lebesgue measure on
	$\mathbb{R}^2$. Let $\H = L^2(\mathbb{R})$. Let us consider the differentiation operator $Af=-if'$ with domain $H^1(\mathbb{R})$ which is a self-adjoint operator of $L^2(\mathbb{R})$ (see \cite[Section 1.3]{Shmudg}). Fix $g \in H^1(\mathbb{R})$ with $\|g\|_2 = 1$, then $\psi_g: (s,t)\in\mathbb{R}^2\to L^2(\mathbb{R})$ defined by $\psi_g(s,t)=-e^{2\pi i t\cdot}(2\pi tg(\cdot-s)+g'(\cdot-s))$ is a continuous weak $A$-frame for $L^2(\mathbb{R})$. Indeed, let $h \in H^1(\mathbb{R})\setminus\{0\}$ and consider   $\Theta_h(f)(t, s) =\int_{\mathbb{R}} f(x)\overline{h(x -s)}e^{-2\pi i tx} dx=\ip{f}{\theta_h(t, s)}_2
	$, $t, s \in \mathbb{R}$, the short-time Fourier transform of $f\in L^2(\mathbb{R})$ with respect to the window $h$, with $\theta_h(f) : \mathbb{R}^2 \to H^1(\mathbb{R})\subset
	L^2(\mathbb{R})$ defined by
	$\theta_h(t, s) =e^{2\pi i t\cdot}h(\cdot -s) 
	$, $t, s \in \mathbb{R}$,
	we have the well-known identity for any $f\in L^2(\mathbb{R})$ (see \cite[Proposition 11.1.2]{ole})
	$$\int_{\mathbb{R}}\int_{\mathbb{R}}
		| \ip{f}{\theta_h(s, t)}
	|_2^2dsdt = \| f \|_2^2\|h\|_2^2,$$
hence, if $\|h\|_2 = 1$, then $\theta_h$ is a continuous Parseval
	frame in $L^2(\mathbb{R})$ (see \cite[Example 4.3]{Bownik}). Hence, $\psi_g=A\theta_g$ is a continuous weak $A$-frame. 
	\enex
	
	\beex Let $X = \R$ and let $\mu$ be the Lebesgue measure on $\mathbb{R}^2$.
	Let  $\H = L^2(0,1)$  and let $\mathcal{I}_{(0,1)}$ be the identity of $L^2(0,1)$. Let us consider the differentiation operator $Af=-if'$ with domain $H^1(0,1)$ which is a densely defined closed operator of $L^2(0,1)$ (see \cite[Section 1.3]{Shmudg}).  The function  $\psi:t\in \R\to\psi_t\in L^2(0,1)$ with $\psi_t=2\pi te^{2\pi i t\cdot }\mathcal{I}_{(0,1)}$ is a continuous weak $A$-frame for  $L^2(0,1)$. Indeed, as proved in \cite[Example 4.2]{Bownik},	the function $\theta:t\in \R\to\theta_t\in H^1(0,1)\subset L^2(0,1)$ such that $\theta_t := e^{2\pi i t\cdot }\mathcal{I}_{(0,1)}$   is a Parseval function in $L^2(0,1)$. Hence $\psi=A	\theta$ is a continuous weak $A$-frame for  $L^2(0,1)$.
	\enex

	\begin{prop}\label{prop: frames for a product}
		Let $A$ be a densely defined operator on $\H$ and $\psi$ be a continuous weak $A$-frame for $\H$ with lower bound $\alpha>0$.
		 If $F\in\B(\H)$ is such that the domain $\D(AF)$ is dense, then $\psi$
			is a continuous weak $AF$-frame for $\H$ too, with lower bound $\alpha\|F^*\|^{-2}$.		
	\end{prop}\begin{proof}By hypothesis there exists $\alpha>0$  such that for every $f\in \D(A^*)$
	\begin{equation*} \alpha\|A^*f\|^2\leq\int_X
	|\ip{f}{\psi_x}|^2 d\mu(x) <\infty.
	\end{equation*} The adjoint $(AF)^*$ is well defined and  $F^*A^*=(AF)^*$ by \cite[Theorem 13.2]{Rudin FA}.
		
		Hence, for every $h\in\D((AF)^*)=\D(F^*A^*)$
		\begin{eqnarray*} \|(AF)^*h\|^2&=&\|F^*A^*h\|^2 \leq \|F^*\|^2\|A^*h\|^2\\&\leq&\frac{1}{\alpha}\|F^*\|^2\int_X
			|\ip{h}{\psi_x}|^2 d\mu(x)<\infty\end{eqnarray*} since $h\in\D(F^*A^*)=\D(A^*)$.
	\end{proof}

	\begin{prop}\label{prop: scales}
		Let $A$ be a self-adjoint operator  and $\psi:x\in X\to \psi_x\in\D(A)\subset\H$  a  continuous weak $A$-frame for $\H$ with lower bound $\alpha$, then $A\psi$
		is a continuous weak $A^2$-frame for $\H$ with the same lower bound $\alpha$. Moreover, if $\psi:x\in X\to \psi_x\in\bigcap_{k=1}^n\D(A^k)\subset\H$, then $A^n\psi$
		is a continuous weak $A^{n+1}$-frame for $\H$, for every fixed $n\in\mathbb{N}$, with the same lower bound $\alpha$. In particular, if  $\psi:x\in X\to \psi_x\in\bigcap_{n\in\mathbb{N}}\D(A^n)\subset\H$ is a  continuous weak $A$-frame for $\H$ with lower bound $\alpha$, then $A^n\psi$
		is a continuous weak $A^{n+1}$-frame for $\H$, for every $n\in\mathbb{N}$, with the same lower bound $\alpha$.
	\end{prop}\begin{proof}By hypotheses $A^2$ is self-adjoint with dense domain $\D(A^2)\subset\D(A)$ and there exists $\alpha>0$  such that for every $f\in \D(A)$
		\begin{equation*}\alpha \|Af\|^2 \leq \int_X
		|\ip{f}{\psi_x}|^2 d\mu(x) <\infty.
		\end{equation*}Hence, for every $h\in\D(A^2)$
		\begin{eqnarray*} \|A^2h\|^2&=&\|A(Ah)\|^2 \leq \frac{1}{\alpha}\int_X
			|\ip{Ah}{\psi_x}|^2 d\mu(x) \\&=&\frac{1}{\alpha}\int_X
			|\ip{h}{A\psi_x}|^2 d\mu(x)<\infty\end{eqnarray*} since $Ah\in\D(A)$.\\  Fix now an arbitrary $n\in\mathbb{N}$. If $\psi:x\in X\to \psi_x\in\D(A^n)\subset\H$, then, as before, by hypotheses both $A^n$ and $A^{n+1}$ are self-adjoint with dense domain $\D(A^{n+1})\subset\D(A^n)\subset\D(A)$ and  for every $h\in\D(A^{n+1})$
		\begin{eqnarray*} \|A^{n+1}h\|^2&=&\|A(A^nh)\|^2 \leq \frac{1}{\alpha}\int_X
			|\ip{A^nh}{\psi_x}|^2 d\mu(x) \\&=&\frac{1}{\alpha}\int_X
			|\ip{h}{A^n\psi_x}|^2 d\mu(x)<\infty\end{eqnarray*} being $A^nh\in\D(A)$. The last sentence in the Theorem is now obvious.
	\end{proof}

	\bedefi Let $A$ be a densely defined operator and $\psi:x\in X\to\psi_x\in\H$, then a function $\phi:x\in X\to\phi_x\in\H$ is called a {\it  weak $A$-dual of} $\psi$  if 	$$
	\ip{A h}{u}=\int_X  \langle h | \phi_x\rangle \ip{\psi_x }{u}d\mu(x),\qquad \forall h\in \D(A), u\in \D(A^*).
	$$
	\findefi
	
	The  weak $A$-dual $\phi$ of  $\psi$ is not unique, in general. \beex
	Let us see two examples. Let $A$ be a densely defined operator on a separable Hilbert space
	$\H$.
	\begin{enumerate}
		\item[i)](See proof of Theorem \ref{thm; existence of cont weak atomic systs unb}) 	Let $(X,\mu)$ be a $\sigma$-finite measure space. 
		Let $\{X_n\}_{n\in\mathbb{N}}$ be a covering of $X$ made up of countably many measurable disjoint sets of finite measure. Without loss of generality we suppose that $\mu(X_n)>0$ for every $n\in\mathbb{N}$. Let $\{e_n\}\subset\D(A)$ be an orthonormal basis of $\H$ 
		 and consider $\psi$, with $\psi_x= \frac{Ae_n}{\sqrt{\mu(X_n)}}$, $x\in X_n, \forall n\in\mathbb{N}$,  then one can take $\phi$ with $\phi_x=\frac{e_n}{\sqrt{\mu(X_n)}}$, $x\in X_n, \forall n\in\mathbb{N}$.
		\item[ii)] If  $\psi:=A\zeta$, where $\zeta:x\in X\to\zeta_x\in\D(A)\subset\H$ is a continuous frame for $\H$, then one can take as $\phi$ any  dual frame of $\{\zeta_x\}$.
	\end{enumerate}
	\enex

\subsection{Frame-related operators of continuous weak $A$-frames}\label{subsec: frame related operators} In this subsection we  will establish some properties of the analysis, synthesis and (generalized) frame operators of a continuous weak $A$-frame with 
$A$ a densely defined operator.  A theorem of characterization for a continuous weak $A$-frame is also given. \\

Consider the sesquilinear form $\Psi$ defined in \eqref{defn: Omega}, then we can prove the following 
\begin{prop}\label{prop: Omega densely def}
	Let $A$ be a densely defined operator and $\psi$  a continuous weak $A$-frame, then $\D(A^*)\subset\D(\Psi)$. Moreover, if $A$ is closable, then $\Psi$ is densely defined.\end{prop}
\begin{proof}By hypotheses and definitions  $\D(A^*)\subset\D(\Psi)$. If $A$ is closable, then $\D(A^*)$ is dense and this concludes the proof.	\end{proof} However, in general $\D(A^*)\subsetneq \D(\Psi)$.

\begin{cor} Let $A$ be a closable and densely defined operator,  $\psi$ a continuous weak $A$-frame, then  the synthesis operator $C_\psi^*$   is closed. \end{cor} \begin{proof} By Proposition \ref{prop: Omega densely def}, the domain $\D(C_\psi )=\D(\Psi)$ of the closed operator $C_\psi $ is dense, hence $C_\psi^*$ is closed and densely defined.\end{proof}

\berem\label{rem: sesquil form} For what has been established until now, if $A$ is closable and densely defined and $\psi$ is a continuous weak $A$-frame, by \eqref{eq: Omega closed} the sesquilinear form $\Psi$ is a {\em densely defined}, nonnegative closed form.  Then there exists the generalized frame operator $S_\Psi$ of $\psi$ defined as in \eqref{def: Frame opera} and the analysis operator $C_\psi $ is closed and densely defined. Moreover, one has
\begin{equation*}
\alpha \|A^* f\|^2\leq\int_X
|\ip{f}{\psi_x }|^2d\mu(x)=\|C_\psi  f\|_2^2=\left \|S_\Psi^\frac{1}{2}f\right \|^2, \qquad\forall f\in \D(A^*).
\end{equation*}	
	\enrem

\begin{cor} Let $(X,\mu)$ be a $\sigma$-finite measure space, $A$  a closable, densely defined operator,  $\psi$ a continuous weak $A$-frame for $\H$. Then  the generalized frame operator $S_\Psi$ of $\psi$ is self-adjoint and the frame operator $S_\psi $  is closable. \end{cor} \begin{proof} By Proposition \ref{prop: Omega densely def}, the domain $\D(\Psi)$ is dense, hence the thesis follows  by Proposition \ref{prop: S subset C^*C}.\end{proof}

	\begin{prop}		Let $A$ be densely defined and closable, $A^*$ injective and $\psi$ a continuous weak $A$-frame for $\H$. Then $C_\psi $ is injective on $\D(A^*)$.
	\end{prop}\begin{proof}The proof is straightforward once observed that in our hypotheses $\alpha\|A^*f\|^2\leq\|C_\psi f\|_2^2$ for every $f\in\D(A^*)$ and some $\alpha>0$.  \end{proof}

The following is a theorem of characterization for continuous weak $A$-frames.

\begin{thm}\label{thm: A=RM supset DM} Let $(X,\mu)$ be a $\sigma$-finite measure space, $A$ a closed densely defined operator and $\psi:x\in X\to \psi_x\in\H$.
	Then the following statements are equivalent.
	\begin{enumerate}
		\item[i)] $\psi$ is  a continuous weak $A$-frame for $\H$;
		\item[ii)] for every $f\in\D(A^*)$, the map $x \to \ip{f}{\psi_x}$ is a measurable function on $X$ and there exists a closed densely defined extension $R$ of $C_\psi^*$, with $\D(R^*)\supset\D(A^*)$, such that $A=RM$ for some $M\in \B(\H,L^2(X,\mu))$.
	\end{enumerate}
\end{thm}
\begin{proof}
	$i)\Rightarrow ii)$ Consider  $B:\D(A^*)\to L^2(X,\mu)$ given by $(B f)(x) = \ip{f}{\psi_x }$, $\forall f\in\D(A^*), x\in X$ which is a restriction of the analysis operator $C_\psi $.
	Since $C_\psi $ is closed, $B$ is closable. $B$ is also densely defined since $\D(A^*)$ is dense. \\
	We apply Lemma \ref{doug gen unb} to $T_1:=A$ and $T_2:=B$ noting that $\|Bf\|_2^2=\int_X
	|\ip{f}{\psi_x }|^2d\mu(x)$. There exists $M\in \B(\H,L^2(X,\mu))$ such that    $A=B^*M$.  Then the statement is proved taking $R=B^*$, indeed $R=B^* \supseteq C_\psi^*$ and  $\D(R)\supset\D(C_\psi^*)$ is dense because $C_\psi $ is closed and densely defined. Note that we have $\D(A^*)=\D(R^*)$ indeed $\D(R^*)=\D(\overline{B})$, $$\D(A^*)\subset\D(\overline{B})=\D(M^*\overline{B})\subset\D((B^*M)^*)=\D(A^*),$$ hence in particular $B$ is closed. \\
	$ii)\Rightarrow i)$ We have $\D(A^*)=\D(R^*)$ indeed  $$\D(A^*)\subset\D(R^*)=\D(M^*R^*)\subset\D((RM)^*)=\D(A^*).$$ For every $f\in\D(A^*)=\D(R^*)$ $$\|A^* f\|^2=\|M^*R^*f\|^2\leq\|M^*\|^2\|R^*f\|^2=\|M^*\|^2\int_X|\ip{f}{\psi_x}|^2d\mu(x)<\infty$$	being $R^*\subset C_\psi $. This proves that $\psi$ is a continuous weak $A$-frame.\end{proof}

\subsection{Atomic systems for unbounded operators $A$ and their relation with $A$-frames}\label{subsec: cont atomic syst}
	
	Now we define our generalization to the continuous case and to unbounded operators of the notion of atomic system for $K$, with $K\in\B(\H)$ \cite{gavruta}. 
	
	\bedefi\label{def: continuous weak atomic system for A} Let $A$ be a  densely defined operator on $\H$. A {\it continuous weak atomic system for $A$} is a function  $\psi:x\in X\to\psi_x\in\H$ such that  for all $f\in\D(A^*)$, the map $x \to \ip{f}{\psi_x}$ is a measurable function on $X$ and
	\begin{itemize}
		\item[i)] $\int_X
		|\ip{f}{\psi_x }|^2d\mu(x)<\infty$,
		for every $f\in\D(A^*)$;
		\item[ii)] there exists $\gamma>0$ such that, for every $f\in\D(A)$, there exists $a_f\in L^2(X,\mu)$, with
		$\|a_f\|_2=\left( \int_X |a_f(x)|^2d\mu(x)\right) ^{1/2}\leq \gamma\|f\|$ and
		\begin{equation}\label{eq: continuous weak atomic system for $A$}
		\ip{Af}{u}=\int_X  a_f(x) \ip{\psi_x }{u}d\mu(x), \qquad \forall u\in \D(A^*).
		\end{equation}
	\end{itemize}
	\findefi

	\berem\label{Lebesgue integrability} If $\psi$ is a continuous weak atomic system for a densely defined operator $A$ then, for every $f\in \D(A)$ and for every $u\in \D(A^*)$ the function $g_f^u(x)= a_f(x) \ip{\psi_x }{u}$  in \eqref{eq: continuous weak atomic system for $A$} is  $\mu$-integrable. Indeed it is \emph{absolutely} integrable: fix any $f\in \D(A)$,  $u\in \D(A^*)$, then by Schwarz inequality $$\int_X |a_f(x) \ip{\psi_x }{u}| d\mu(x)\leq \|a_f\|_2 \left( \int_X  |\ip{\psi_x }{u}|^2d\mu(x)\right )^{1/2}< \infty, $$
	where the last inequality follows from both conditions in Definition \ref{def: continuous weak atomic system for A}.
	
	\enrem
	
	The next theorem guarantees the existence of continuous weak atomic systems for densely defined operators on $\H$.

		\begin{thm}\label{thm; existence of cont weak atomic systs unb}
		Let $(X,\mu)$ be a $\sigma$-finite measure space. 
		Let $\H$ be a separable Hilbert space and $A$ a densely defined operator on $\H$. Then there exists a continuous weak atomic system for $A$.
	\end{thm}
	\begin{proof}
		
		Let $\{e_n\}_{n\in\mathbb{N}}\subset \D(A)$ be an othonormal basis for $\H$. Then, every $f\in\H$ can be written as $f=\sum_{n=1}^\infty
		\ip{f}{e_n}e_n$.
		For all $n\in\mathbb{N}$ denote with $\psi_n=Ae_n$. Let $\{X_n\}_{x\in\mathbb{N}}$ be a covering of $X$ made up of countably many measurable disjoint sets of finite measure. It is not restrictive supposing that $\mu(X_n)>0$ for every $n\in\mathbb{N}$. Then we define
 \begin{equation*}
		\psi_x:= \frac{\psi_n}{\sqrt{\mu(X_n)}}, \quad  x\in X_n,  n\in\mathbb{N}.
		\end{equation*}
		
		For every $f\in\H$ the map $x\in X \to \ip{f}{\psi_x}\in\mathbb{C}$ is measurable because it is a step function.
		
		Moreover, for every $f\in\D(A
		^*)$
		\begin{eqnarray*}
			\|A^*f\|^2&=&\sum_{n=1}^\infty
			\left|\ip{A^*f}{e_n}\right|^2=\sum_{n=1}^\infty
			\left|\ip{f}{Ae_n}\right|^2\\
			&=&\sum_{n=1}^\infty \left|\ip{f}{\psi_n}\right|^2=\sum_{n=1}^\infty\int_{X_n}
			|\ip{f}{\psi_x}|^2d\mu(x)
			\\&=&	\int_X |\ip{f}{\psi_x}|^2d\mu(x)<\infty.
		\end{eqnarray*}

		Now, for all $f\in\D(A^*)$, take $a_f$ as the step function defined as follows:
		
		\begin{equation*}
		a_f(x) :=\frac{\ip{f}{e_n}}{\sqrt{\mu(X_n)}}, \quad  x\in X_n,  n\in \mathbb{N}.
		\end{equation*}  Then, for all $f\in\D(A^*)$,   $a_f\in L^2(X,\mu)$, with
		\begin{eqnarray*}
			\|a_f\|_2^2&=& \int_X |a_f(x)|^2d\mu(x) =\sum_{n=1}^\infty\int_{X_n}
			\frac{\left |\ip{f}{e_n}\right |^2}{\mu(X_n)}d\mu(x)\\&=& \sum_{n=1}^\infty
			|\ip{f}{e_n}|^2=\|f\|^2,
		\end{eqnarray*}
			and for every $f\in\D(A)$, $u\in\D(A^*)$
		\begin{eqnarray*}\ip{Af}{u}&=&\ip{\sum_{n=1}^\infty\ip{f}{e_n}Ae_n}{u}\\
                                   &=&\sum_{n=1}^\infty\ip{f}{e_n}\ip{Ae_n}{u}\\
                                  &=&\sum_{n=1}^\infty\int_{X_n}
				\frac{\ip{f}{e_n}}{\sqrt{\mu(X_n)}}\frac{\ip{Ae_n}{u}}{\sqrt{\mu(X_n)}}d\mu(x)=\int_X a_f(x)\ip{\psi_x }{u}d\mu(x)\end{eqnarray*}
			Therefore $\psi$ is a continuous weak atomic system for $A$.
	\end{proof}

The following theorem gives a characterization of continuous weak atomic systems for $A$ and continuous weak $A$-frames.
	
	\begin{thm}
		\label{th_char_continuous weak_A_frame}
		Let $(X,\mu)$ be a $\sigma$-finite measure space,  $\psi:x\in X\to \psi_x\in$ and $A$  a closable densely defined operator. Then the following statements are equivalent.
		\begin{itemize}
			\item[i)] $\psi$ is a continuous weak atomic system for $A$;
			\item[ii)] $\psi$ is a continuous weak $A$-frame;
			\item[iii)] $\int_X
			|\ip{f}{\psi_x }|^2d\mu(x)<\infty$
			for every $f\in\D(A^*)$ and there exists a Bessel   weak $A$-dual $\phi$ of $\psi$.
		\end{itemize}
	\end{thm}
	\begin{proof} $i)\Rightarrow ii)$
		For every  $f\in \D(A^*)$  by the density of $\D(A^*)$ we have
		\begin{eqnarray*}\|A^* f\|&=&\sup_{h\in \H,\|h\|=1}\left|\ip{A^*
				f}{h}\right|	=\sup_{h\in \D(A),\|h\|=1}\left|\ip{A^*
				f}{h}\right|\\	&=&
			\sup_{h\in\D(A),\|h\|=1}|\ip{f}{Ah}|
			\\&=&\sup_{h\in \D(A),\|h\|=1}\left|\int_X
			\overline{a_h(x)}\ip{f}{ \psi_x }d\mu(x)\right|\\ &\leq&
			\sup_{h\in \D(A),\|h\|=1}\left(\int_X
			|a_h(x)|^2d\mu(x)\right)^{1/2}\left(\int_X |\ip{f}{
				\psi_x }|^2d\mu(x)\right)^{1/2}\\&\leq& \gamma\left(\int_X |\ip{f}{
				\psi_x }|^2d\mu(x)\right)^{1/2}<\infty,\end{eqnarray*}
		for some $\gamma>0$, the last two inequalities are due to the fact that $\psi$ is a continuous weak atomic system for $A$. \\
	$ii)\Rightarrow iii)$ Following the proof of Theorem \ref{thm: A=RM supset DM}, there exists $M\in \B(\H,L^2(X,\mu))$ such that    $\overline{A}=B^*M$, with $B:\D(A^*)\to L^2(X,\mu)$ a closable, densely defined operator which is a restriction of the analysis operator $C_\psi $. \\ 	 
	By the Riesz representation theorem, for every $x\in X$ there exists a unique vector $\phi_x\in\H$ such that $(Mh)(x)=\ip{h}{\phi_x}$ for every $h\in \H$. The function $\phi:x\in X\to\phi_x\in\H$ is Bessel. Indeed,  	\begin{eqnarray*}\int_X
		|\ip{f}{\phi_x}|^2 d\mu(x)& \leq&  \int_X |(Mf)(x)|^2d\mu(x)\\&=&\|Mf\|_2^2\leq\|M\|_{L^2}^2\|f\|^2, \qquad\forall f\in\H.
	\end{eqnarray*}	
Moreover,  for $h\in \D(A), u\in \D(A^*)=\D(B)$
\begin{eqnarray*}
\ip{A h}{u}&=&\ip{\overline{A} h}{u}=\ip{B^* M h}{u}=\ip{M h}{B^{**} u}_2\\
&=&\ip{M h}{Bu}_2=\int_X\ip{h}{\phi_x}\ip{\psi_x}{u}d\mu(x).
\end{eqnarray*}
		$iii)\Rightarrow i)$ It suffices to take $a_f:x\in X\to a_x(f)=\ip{f}{\phi_x}\in\mathbb{C}$ for all $f\in \D(A)$. Indeed $a_f\in L^2(X,\mu)$ and, for some $\gamma> 0$, we have $\int_{X} |a_x(f)|^2 d\mu(x) = \int_{X}|\ip{f}{\phi_x} |^2 d\mu(x) \leq \gamma  \|f\|^2$  since $\phi$ is a Bessel function, moreover, $\ip{Af}{u}=\int_X  a_f(x) \ip{\psi_x }{u}d\mu(x)$, for $f\in \D(A), u\in \D(A^*)$.
	\end{proof}

The proof of Theorem \ref{th_char_continuous weak_A_frame}  suggests the following \begin{prop}
	Let $\D\subset\H$ be dense,  $\psi:x\in X\to\psi_x\in\H$ be such that\begin{itemize}
		\item[i)] for every $f\in\D$, the map $x \to \ip{f}{\psi_x}$ is a measurable function on $X$
		\item[ii)] $\int_X|\ip{f}{\psi_x}|^2d\mu(x)<\infty$ for every $f\in\D$.\end{itemize}   If $M\in\B(\H, L^2(X,\mu))$ and $x\in X$ denote by $\phi_x$ the unique vector of $\H$ such that $(Mh)(x)=\ip{h}{\phi_x}$ for every $h\in \H$. Then, there exists a closed, densely defined operator $A_M$ such that $\psi$ is a continuous weak atomic system for $A_M$ and  $\phi:x\in X\to \phi_x\in\H$ is a Bessel function  which is a  weak $A_M$-dual of $\psi$.
\end{prop}
\begin{proof}
	Let us consider the operator $B:\D\to L^2(X,\mu)$ defined for every $f\in\D$ by $(Bf)(x)=\ip{f}{\psi_x}$, $\forall x\in X$ which is a restriction of the analysis operator $C_\psi $. Since $B$ is densely defined, then $B^*$, the adjoint of $B$, is well defined. Now fix any    $M\in\B(\H, L^2(X,\mu))$, for every $h\in\H$ and any $x\in X$ by the Riesz representation theorem there exists a function $\phi:x\in X\to\phi_x\in\H$ such that $(Mh)(x)=\ip{h}{\phi_x}$.  By the same calculations than in Theorem \ref{th_char_continuous weak_A_frame}, $\phi$ is a Bessel function. Consider the closed operator $E=B^*M$, then $E^*\supset M^*B^{**}\supset M^*B$ and define $F=E^*_{\upharpoonright\D}=M^*B$ which is closable and densely defined. Then $\D(F^*)$ is dense and $\forall u\in \D=\D(F)$ and $\forall h\in \D(F^*)$ we have \begin{eqnarray*}
		\ip{F^* h}{u}&=&\ip{h}{Fu}=\ip{h}{M^*Bu}=\ip{M h}{Bu}_2\\
		&=&\int_X\ip{h}{\phi_x}\ip{\psi_x}{u}d\mu(x).
	\end{eqnarray*} It suffices now to take $A_M=F^*$.
\end{proof}
\medskip

If $\mathcal{R}(A)$ is weakly decomposable, then $\mathcal{R}(A^*)$ is weakly decomposable too.
	
	\begin{prop} \label{rem_adj_continuous weak}
	Let $ A$ be a densely defined operator on  $\H$, $\psi$  a continuous weak atomic system for $A$ and $\phi$ a Bessel  weak $A$-dual of $\psi$. Then, the adjoint $A^*$ of $A$ admits a weak decomposition 	and \begin{equation*}
	\label{A^*_dual}
	\ip{A^* u}{h}=\int_X \ip{u}{\psi_x }\ip{\phi_x}{h}d\mu(x),\qquad \forall u\in \D(A^*), \forall h\in\D(A).
	\end{equation*} 	\end{prop}\begin{proof}Fix any $u\in\D(A^*)$ then, for every  $h\in \D(A)$ \begin{eqnarray*}\label{exp_A^*_continuous weak}
	\ip{A^* u}{h}&=&\ip{u}{Ah}=\int_X  \overline{\langle h | \phi_x\rangle \ip{\psi_x }{u}}d\mu(x)\\  \nonumber&=& \int_X \ip{u}{\psi_x } \langle \phi_x| h\rangle d\mu(x).
	\end{eqnarray*}
\end{proof}

\berem In the discrete case, i.e. for $X=\mathbb{N}$ and $\mu$ a
 counting measure, albeit a strong decomposition of $A$ is still not guaranteed in general, the adjoint $A^*$ admits a strong decomposition (see \cite[Remark 3.13]{BC}) in the sense that $$
 A^*f=\sum_{n=1}^\infty \langle f | \psi_n\rangle\phi_n,\quad \forall f\in \D(A^*) $$ with $\{\phi_n\}$ a Bessel weak $A$-dual of the weak $A$-frame $\{\psi_n\}$ .\enrem

	\berem Contrarily to the case in which the operator is in $\B(\H)$, given a closed densely defined operator $A$ on $\H$ and a continuous weak $A$-frame $\psi$,  a  weak $A$-dual $\phi$ of $\psi$ is not a continuous weak $A^*$-frame, in general. For example, if $A$ is unbounded and $\phi$ is also a  Bessel function,  from the inequality
	$$
	\alpha\|Af\|^2\leq \int_X |\ip{f}{\phi_x}|^2d\mu(x), \qquad \forall f\in \D(A)
	$$
	with $\alpha>0$, we obtain that  $A$ is bounded, a contradiction.  \enrem
	
We conclude this section by proving that, under suitable hypotheses, we can weakly decompose the domain of $A^*$ by means of a continuous weak $A$-frame. 
		\begin{thm}\label{thm: interchang unb}
			Let $A$ be a closed densely defined operator with $\mathcal{R}(A)=\H$ and $A^\dag$  the  pseudo-inverse  of $A$. Let $\psi$ be a continuous weak  $A$-frame and $\phi $  a Bessel  weak $A$-dual of $\psi$. Then, the function $\vartheta$  with $\vartheta_x:=(A^\dag)^* \phi_x\in\H$, for every  $x\in X$, is Bessel and  every $u\in \D(A^*)$ can be weakly decomposed as follows\begin{equation*}\label{interch u}
			\ip{f}{u}=\int_X \langle { f | \vartheta_x\rangle }\ip{\psi_x }{u} d\mu(x)\qquad \forall f \in\H, u\in \D(A^*).
		\end{equation*}
		\end{thm}
		\begin{proof}		
			By Lemma \ref{lem: pseudoinverse unb} there exists a unique pseudo-inverse $A^\dag\in\B(\H)$  of $A$
			such that $f = AA^\dag f$, $f \in\H$. Then,
			$$		\ip{f}{u}=\ip{AA^\dag f}{u}=\int_X \langle A^\dag f | \phi_x\rangle \ip{\psi_x }{u} d\mu(x)\qquad \forall f \in\H, u\in \D(A^*).
			$$
			Consider the adjoint $(A^\dag)^*\in\B(\H)$ of $A^\dag$ and define $\vartheta_x:=(A^\dag)^* \phi_x\in\H$, for every  $x\in X$. Then, for any $f\in\H$, we have		\begin{eqnarray*}\ip{f}{u}&=&\int_X  \langle { f | (A^\dag)^* \phi_x\rangle} \ip{\psi_x }{u}d\mu(x)\\&=&\int_X \langle { f |\vartheta_x\rangle} \ip{\psi_x }{u} d\mu(x),\qquad \forall u\in \D(A^*)
			\end{eqnarray*} and
			
			\begin{eqnarray*}\int_X |\ip{f}{\vartheta_x}|^2 d\mu(x)&=&\int_X \left |\ip{f}{(A^\dag)^* \phi_x}\right |^2d\mu(x)=
				\int_X \left |\ip{A^\dag f}{\phi_x}\right |^2d\mu(x)\\&\leq&\gamma\| A^\dag f\|^2\leq\gamma \|A^\dag \|^2\|f\|^2\end{eqnarray*} for some $\gamma>0$ since $\phi$ is Bessel and $A^\dag$ is bounded.
			Hence, $\vartheta:x\in X\to\vartheta_x\in\H$
			is a  Bessel function.
		
					\end{proof}
	\berem In the discrete case the decomposition of the domain of $\D(A^*)$ is strong \cite{BC}.\enrem

\section{Continuous atomic systems for bounded operators \\ between different Hilbert spaces}\label{sec: cont K-frames}
In this section we  introduce our second approach to the generalization of the notion of (discrete)  atomic system for $K\in\B(\H)$ and of $K$-frame in \cite{gavruta}, to unbounded operators in a Hilbert space in the continuous framework. Since  a closed densely defined operator in a Hilbert space $A:\D(A)\to\H$ can be seen as a {\em bounded} operator $A:\H_A\to \H$ between two different Hilbert spaces (with $\H_A$ the Hilbert space $\D(A)[\|\cdot\|_A]$ where $\|\cdot\|_A$ is the graph norm), 
 before introducing new notions, we 
 put the main definitions and results in \cite{AROR,gavruta} for $K\in\B(\H)$ in terms of bounded operators from a  Hilbert space into another. Later, in Section  \ref{subsec: cont atomic sys for unbounded op}, we return to the operator $A:\H_A\to \H$.\\

Let $\H$, $\J$ be two Hilbert spaces with inner products $\ip{\cdot}{\cdot}_\H,\ip{\cdot}{\cdot}_\J$  and induced norms $\|\cdot\|_\H,\|\cdot\|_\J$, respectively. We denote by $\B(\J,\H)$ the set of bounded linear operators from $\J$ into $\H$. For any $K\in \B(\J,\H)$  we denote by $K^*\in \B(\H,\J)$ its adjoint.

\begin{defn}\label{cont atom syst for K}	Let $K\in\B(\J,\H)$. The function $\psi:x\in X\to\psi_x\in\H$  is a continuous atomic system for $K$ if  for all $h\in\H$, the map $x \to \ip{h}{\psi_x}_\H$ is a measurable function on $X$ and\begin{itemize}
		\item[i)] $ \psi$ is Bessel function \item[ii)] there exists $\gamma>0$ such that for all $f\in\J$ there exists $a_f\in L^2(X,\mu)$, with
		$\|a_f\|_2=\left( \int_X |a_f(x)|^2d\mu(x)\right) ^{1/2}\leq \gamma\|f\|_\J$ and for every $g\in\H$               $$\ip{Kf}{g}_\H=\int_X a_f(x)\ip{\psi_x}{g}_\H d\mu(x).$$
	\end{itemize}
\end{defn}
If $\J=\H$ and $\mu$ is a counting measure, then the previous notion reduces to the notion of atomic system for $K\in\B(\H)$ in  \cite{gavruta}.

\beex\label{cont exist thm}Let $K\in \B(\J,\H)$.
Every continuous frame $\psi$ for $\H$ is a continuous atomic system for $K$. Indeed, if $\phi$ is a  dual frame of  $\psi$, then for every $h\in\H$
\begin{equation*}
\ip{Kf}{h}_\H=\int_X \ip{Kf}{\phi_x}_\H \ip{\psi_x}{h}_\H d\mu(x), \qquad\forall f\in \J
\end{equation*}
and Definition \ref{cont atom syst for K} is satisfied by taking $a_f(x)=\ip{Kf}{\phi_x}_\H$ for $f\in \J$.

\enex

\beex\label{ex: cont atomic syst} Let   $K\in \B(\J,\H)$  and $\xi:x\in X\to\xi_x\in\J$ a continuous frame for $\J$ with  dual  frame $\vartheta:x\in X\to\vartheta_x\in\J$, then for all $f,g\in \J$
$$\ip{f}{g}_\J=\int_X \ip{f}{\vartheta_x}_\J \ip{\xi_x}{g}_\J d\mu(x),$$  hence, for every $h\in\H$ $$ \ip{Kf}{h}_\H=\ip{f}{K^*h}_\J=\int_X \ip{f}{\vartheta_x}_\J \ip{K\xi_x}{h}_\H d\mu(x) .$$
Thus the function $\psi=K\xi$ is
a continuous atomic system for $K$, taking $a_f(x):=\ip{f}{\vartheta_x}_\J$.
\enex

In the discrete case, the decomposition of   $\mathcal{R}(K)$, the range of $K$, is strong \cite{BC}.\\

We give a result of existence of a continuous atomic system for a bounded operator.
\begin{thm}\label{thm: existence bdd}
	Let $(X,\mu)$ be a $\sigma$-finite measure space, 
	$\J$  a separable Hilbert space and $K\in\B(\J,\H)$. Then there exists a continuous atomic system for $K$.
\end{thm}
\begin{proof} With the same notation than in Theorem \ref{thm; existence of cont weak atomic systs unb} we have  that
	\begin{eqnarray*}
		\int_X |\ip{h}{\psi_x}_\H|^2d\mu(x)&=&\sum_{n=1}^\infty\int_{X_n}
		|\ip{h}{\psi_x}_\H|^2d\mu(x)=
		\sum_{n=1}^\infty \left|\ip{h}{\psi_n}_\H\right|^2\\&=&\sum_{n=1}^\infty
		\left|\ip{h}{Ke_n}_\H\right|^2=\sum_{n=1}^\infty
		\left|\ip{K^*h}{e_n}_\J\right|^2\\&=&\|K^*h\|_\J^2\leq\|K^*\|_{\H,\J}^2\|h\|_\H^2,
	\end{eqnarray*} 
	where the last equality is due to the Parseval identity. The thesis follows from Theorem \ref{thm; existence of cont weak atomic systs unb}, with slight modifications due to the fact that $K\in\B(\J,\H)$.
\end{proof}

\bedefi \label{def: cont_K_frame} Let  $K\in\B(\J,\H)$. A function $\psi:x\in X\to\psi_x\in\H$ is called a \emph{continuous $K$-frame for $\H$} if for all $h\in\H$, the map $x \to \ip{h}{\psi_x}_\H$ is a measurable function on $X$ and  there exist $\alpha, \beta>0$  such that for every $h\in \H$
\begin{equation}\label{defn cont K frame for H}\alpha \|K^*h\|_\J^2 \leq \int_X
|\ip{h}{\psi_x}_\H|^2 d\mu(x) \leq \beta \|h\|_\H^2.
\end{equation} The constants $\alpha,\beta$ will be called frame bounds.
\findefi

It is easy to see that if $K\in\B(\J,\H)$ and $\theta$ is a continuous frame for $\J$, then $K\theta$ is a continuous $K$-frame for $\H$. Then we give the following two examples.

\beex Let $X = \R$ and let $\mu$ be the Lebesgue measure.
Let  us identify  $\J=\H = L^2(0,1)$ and let $\mathcal{I}_{(0,1)}$ be the identity of $L^2(0,1)$. Fix any $g\in C(0,1)$, the space of continuous functions on the open interval $(0,1)$  (or also $g\in L^\infty(0,1)$ the space of essentially bounded functions on $(0,1)$), and consider the self-adjoint operator $M_g\in \B(L^2(0,1))$ defined by $
M_gf = gf $ for every $f \in L^2(0,1)$. Then, $\psi_t := ge^{2\pi i t \cdot}\mathcal{I}_{(0,1)}$ is a   continuous $M_g$-frame. Indeed, as proved in \cite[Example 4.2]{Bownik},	the function $\theta:t\in \R\to\theta_t\in L^2(0,1)$ such that $\theta_t := e^{2\pi i t \cdot}\mathcal{I}_{(0,1)}$   is a Parseval function in $L^2(0,1)$, hence $\psi=M_g\theta$  is a   continuous $M_g$-frame.
\enex

\berem If   $\J=\H$ a continuous  $K$-frame $\psi$ is a continuous $K$-g-frame in the sense of \cite[Definition 2.1]{AROR} with $\Lambda_x=\ip{f}{\psi_x}$ for every $f\in\H$, with $x\in X$. If $K\in\B(\J,\H)$, $X=\mathbb{N}$ and $\mu$ is a counting measure, a continuous  $K$-frame clearly reduces to a discrete $K$-frame in the sense of \cite{BC} and, if in addition $\J=\H$, coincides with that of $K$-frame in \cite{gavruta}.\enrem

\begin{prop}\label{prop: product} Let $\H$, $\J$ and $\F$ be Hilbert spaces, $K\in\B(\J,\H)$, $E\in\B(\H,\F)$, $G\in\B(\H,\J)$ and $\psi$ be a continuous $K$-frame for $\H$, then
		\begin{itemize}
		\item[i)]  $E\psi$
		is a continuous $EK$-frame for $\F$.
		
		\item[ii)]  $\psi$
		is a continuous  $KG$-frame for $\H$ too.
	\end{itemize}
\end{prop}\begin{proof}	$i)$	It is a slight modification of the proof in \cite[Theorem 3.4]{AROR}.\\
		$ii)$ It descends from Proposition \ref{prop: frames for a product} with obvious adaptations. 
\end{proof}

A natural consequence is the following (see also  \cite[Corollary 3.5]{AROR})
\begin{cor}
	Let $K\in\B(\H)$ and $\psi$ be a continuous $K$-frame for $\H$, then $\psi$ and $K^n\psi$
	are continuous $K^{n+1}$-frames for $\H$, for every integer $n\geq0$.
\end{cor}

Let us give a characterization of  continuous atomic systems for operators in $\B(\J,\H)$.

\begin{thm}\label{thm: char cont atomic system}
	Let $\psi:x\in X\to\psi_x\in\H$ and  $K\in\B(\J,\H)$. Then the following are equivalent.
	\begin{itemize}
		\item[i)] $\psi$ is a continuous atomic system for $K$;
		\item[ii)] $\psi$ is a continuous $K$-frame for $\H$; 
		\item[iii)] $\psi$ is a Bessel function and there exists a Bessel function $\phi:X\to\J$   such that
		\begin{equation}\label{exp_K}
		\ip{Kf}{h}_\H=\int_X\ip{f}{\phi_x}_\J\ip{\psi_x}{h}_\H d\mu(x)\qquad \forall f\in \J, \forall h\in\H.
		\end{equation}
	\end{itemize}
\end{thm}
\begin{proof} The proof follows from Theorem \ref{th_char_continuous weak_A_frame}, with suitable adjustments,
 recalling that if $\psi$ is a continuous $K$-frame for $\H$, then it is a Bessel function. 	

\end{proof}

As in the discrete case, \bedefi Let $K\in\B(\J,\H)$ and $\psi:x\in X\to\psi_x\in\H$ a continuous $K$-frame for $\H$.  A function $\phi:X\to\H$ as in \eqref{exp_K} is called a {\it  $K$-dual} of $\psi$.\findefi

\beex  In general,  a  $K$-dual  $\phi:x\in X\to\phi_x\in\J$ of a continuous $K$-frame $\psi:x\in X\to\psi_x\in\H$ is not unique. Let us see some examples. 
\begin{enumerate}
	\item[i)]If $\psi=\zeta$, where $\zeta:X\to\H$ is a continuous frame for $\H$, then one can take  $\phi=K^*\xi:X\to\J$ where $\xi:x\in X\to\xi_x\in\H$ is any  dual frame of $\zeta$.
	\item[ii)] If $\psi=K\zeta$, where $\zeta:x\in X\to\zeta_x\in\J$ is a continuous frame for $\J$, then one can take as $\phi$ any dual frame of $\zeta$.
\end{enumerate}
\enex

\berem\label{strong decomp of K^*}  Once at hand a continuous atomic system $\psi$ for $K$, a  Bessel $K$-dual  $\phi:X\to\J$ as in Theorem \ref{thm: char cont atomic system} is a continuous atomic system for $K^*$. Indeed,
\begin{eqnarray*}
	\ip{K^* h}{f}_\J &=&\ip{h}{Kf}_\H= \ \int_{X} \overline{\ip{f}{\phi_x}_\J\ip{\psi_x}{h}_\H} d\mu(x)\\&=&\int_{X} \ip{h}{\psi_x}_\H \ip{\phi_x}{f}_\J d\mu(x)  , \qquad f\in\J, h\in\H.
\end{eqnarray*}
We apply Theorem \ref{thm: char cont atomic system} to $K^*$ and $\phi$ to conclude that $\phi$ is a continuous atomic system for $K^*$.
\enrem

Following  H.G. Feichtinger and T. Werther \cite{FW}, 

\begin{defn}
	Let $\psi:x\in X\to \psi_x\in\H$ be a Bessel function and $\H_0$  a closed subspace of $\H$. The function $ \psi$ is called a continuous family of local
	atoms for $\H_0$ if there exists a family of linear functionals $\{c_x\}$ with $c_x:\H\to\mathbb{C}$ for every $x\in X$, such that  
	\begin{itemize}
		\item[i)] exists $\gamma>0$ with
		$\int_X|c_x(f)|^2d\mu(x)\leq\gamma\|f\|^2,\, \forall f\in \H_0$;
		\item[ii)] $\ip{f}{h} =	\int_{X} c_x(f) \ip{\psi_x}{h} d\mu(x),\, \forall f\in \H_0,  h\in\H$.
	\end{itemize}
\end{defn}

We will say that the pair $\{ \psi_x, c_x\}$ provides an atomic decomposition for $\H_0$ and $\gamma$ will be called an atomic bound of $\{ \psi_x\}$.

If now $K=P_{\H_0}\in\B(\H)$ is  the orthogonal projection on $\H_0$ ($P_{\H_0}=P_{\H_0}^2=P_{\H_0}^*$), a continuous $P_{\H_0}$-frame is a family of continuous local atoms for $\H_0$, similarly to  \cite[Theorem 5]{gavruta}.

\begin{cor}
	Let $\psi:x\in X\to \psi_x\in\H$ be
	a  Bessel function and $\H_0$  a closed subspace of the Hilbert space $\H$. Then the following statements are equivalent.
	\begin{itemize}
		\item[i)] $\{ \psi_x\}$
		is a family of continuous local atoms for $\H_0$;
		\item[ii)] $\psi$ is a continuous atomic system for $P_{\H_0}$;
		\item[iii)] there exists $\alpha > 0$ such that $\alpha\|P_{\H_0}f\|^2\leq\int_X |\ip{f}{\psi_x}|^2d\mu(x)$, $f\in\H$;
		\item[iv)] there exists a  Bessel function  $\phi:x\in X\to \phi_x\in\H$ such that
		$$\ip{P_{\H_0} f}{h} =\int_X \ip{f}{\phi_x}\ip{\psi_x}{h}d\mu(x),$$ for any $f,h\in\H$.\end{itemize}\end{cor}

Not even if $\J=\H$ a Bessel function $\psi:X\to\H$ and a  $K$-dual $\phi:X\to\H$ of its are interchangeable, in general. However, if we strenghten hypotheses on $K$, it can be proved the existence of a function with range in $\H$ which is interchangeable with  $\psi$ in the weak decomposition of $\mathcal{R}(K)\subset\H$  (see also \cite[Theorem 3.2]{AROR}).

\begin{thm}\label{thm: interchang boundd}
	Let $K\in\B(\J,\H)$ with closed range $\mathcal{R}(K)$. Let   $\psi$ be a continuous $K$-frame and $\phi$ a Bessel  $K$-dual of its. Then,\begin{itemize}
		\item[i)]  the function $\vartheta:x\in X\to\vartheta_x\in\H$ with $\vartheta_x:=(K^\dag_{\upharpoonright\mathcal{R}(K)})^*\phi_x\in\H$, for every  $x\in X$, is Bessel for $\mathcal{R}(K)$ and   interchangeable with  $\psi$ 	for any $h\in \mathcal{R}(K)$, i.e.  \begin{equation*}\label{cont interch f bdd}
\ip{h}{f}_\H= 	\int_X\ip{h}{\vartheta_x}_\H\ip{\psi_x}{f}_\H d\mu(x)=
	\int_X\ip{h}{\psi_x}_\H\ip{\vartheta_x}{f}_\H d\mu(x),\, f \in\H;\end{equation*}

		\item[ii)]   $\vartheta$ is a  continuous $K$-frame for $\H$ and $K^* \vartheta$ and $K^* \psi$ are Bessel  $K$-duals of $\psi$ and of $\vartheta$ respectively. In particular,  for every $h\in \H$ \begin{eqnarray}\label{eq: decomp of K}
		\ip{Kf}{h}_\H&=& 	\int_X\ip{f}{K^* \vartheta_x}_\J  \ip{\psi_x}{h}_\H d\mu(x)\\&=&\nonumber
		\int_X
		\ip{f}{K^* \psi_x}_\J  \ip{\vartheta_x}{h}_\H d\mu(x),\qquad \forall f \in\J.	\end{eqnarray}
	\end{itemize}
	\end{thm}
\begin{proof}
	$i)$ See \cite[Theorem 3.2]{AROR} with obvious adjustments.\\ 
	$ii)$ Clearly \eqref{eq: decomp of K} follows from $i)$. The function $\vartheta$ is a continuous $K$-frame for $\H$ by $i)$ and \eqref{eq: decomp of K}, taking for all $f\in\J$, $a_f(x)=\ip{f}{K^* \psi_x}_\J  $, for every $x\in X$. The functions $K^* \vartheta$ and $K^* \psi$ are Bessel for $\J$, indeed  for all $f\in\J$, the maps $x \to \ip{K^\dag_{\upharpoonright\mathcal{R}(K)} Kf}{\phi_x}_\J=\ip{f}{K^*\vartheta_x}_\J$ and $x \to \ip{Kf}{ \psi_x}_\H=\ip{f}{K^*\psi_x}_\J$ are measurable functions on $X$ and \begin{eqnarray*}
		\int_X|\ip{f}{K^* \vartheta_x}_\J|^2d\mu(x)&=&\int_X|\ip{Kf}{ \vartheta_x}_\H|^2d\mu(x)\\
		&\leq&\beta\|Kf\|_\H^2\leq\beta\|K\|_{\J,\H}^2\|f\|_\J^2,\qquad\forall f\in\J\end{eqnarray*}
	for some $\beta>0$.
	Similarly, $K^* \psi$ is Bessel.
	The proof is concluded by using Theorem \ref{thm: char cont atomic system}.
\end{proof}
\berem  Consider the function $\psi:x\in X \to\psi_x\in\H$. In  this section  the frame operator $S_\psi $ of $\psi$ will be  denoted by
\begin{equation*}\label{eq: frame operator}
\ip{S_\psi f}{g}_\H=\int_X \ip{f}{\psi_x}_\H\ip{\psi_x}{g}_\H d\mu(x),\qquad f\in\D(S_\psi ), g\in\H
\end{equation*}where $$\D(S_\psi )=\{f\in \H:\int_X \ip{f}{\psi_x}_\H \psi_xd\mu(x)\,\, \mbox{converges weakly in }\H\}$$ Later on (see Remark \ref{rem: S everyw def}) we will see that, as for continuous $K$-frames with $K\in\B(\H)$, the domain $\D(S_\psi )$ of the frame operator of a continuous $K$-frame with $K\in\B(\J,\H)$ coincides with the whole $\H$ .\\
The analysis operator  of the function $\psi$ will be indicated by $C_\psi :h\in\D(C_\psi )\subset\H\to \ip{h}{\psi_x}_\H\in L^2(X,\mu)$  (strongly) defined, for every $h\in\D(C_\psi )$ and for every $x\in X$, by \begin{equation*}(C_\psi h)(x)=\ip{h}{\psi_x}_\H.\end{equation*} and the synthesis operator of $\psi$ by $C_\psi^*:\D(C_\psi^*)\subset L^2(X,\mu)\to \H$ will be denoted by:
\begin{equation*}
\ip{C_\psi^*a}{h}_\H=\int_{X}a(x)\ip{\psi_x}{h}_\H d\mu(x),\qquad a\in \D(C_\psi^*),\,h\in\H
\end{equation*} where
$$\D(C_\psi^*):=\left \{a\in L^2(X,\mu):\int_Xa(x)\ip{\psi_x}{h}_\H d\mu(x) \mbox{ exists } \forall h\in \H\right \}.$$
\enrem

We can characterize continuous $K$-frames for $\H$ by means of both their frame and synthesis operators.
{\begin{thm}\label{thm: decompos op} Let $(X,\mu)$ be a $\sigma$-finite measure space.
		Let $K\in \B(\J,\H)$ and $\psi:x\in X \to\psi_x\in\H$ such that for all $f\in\H$, the map $x \to \ip{f}{\psi_x}$ is a measurable function on $X$. Then the following statements are equivalent.
		\begin{enumerate}
			\item[i)] $\psi$ is a continuous $K$-frame for $\H$;
			\item[ii)] $C_\psi^*$ is bounded and $\mathcal{R}(K)\subset\mathcal{R}(C_\psi^*)$;\item[iii)] $C_\psi^*$ is bounded and there exists $M\in \B(\J,L^2(X,\mu))$ such that $K=C_\psi^*M$.
		\item[iv)] $S_\psi =C_\psi^*C_\psi \geq\alpha KK^*$ on $\H$, for some $\alpha>0$ \footnote{ i.e.  and $\alpha \ip{KK^*f}{f}_\H\leq
			\ip{S_\psi  f}{f}_\H$ for every $f\in\H$.} and $\psi$ is a  Bessel function for $\H$;	\item[v)]$K=\left(S_\psi^{1/2}\right)U$, for some $U\in\B(\J,\H)$. \end{enumerate}
	\end{thm}
	\begin{proof} 
		$i)\Rightarrow ii)$ The   operator $C_\psi^*$ is bounded by Proposition \ref{prop: syntesis and analysis operator}. Moreover,
		for every $h\in \H$ 		
		\begin{eqnarray*}\alpha \|K^*h\|_\J^2 &\leq& \int_X
			|\ip{h}{\psi_x}_\H|^2 d\mu(x) =\|C_\psi 
			h\|_2^2.
		\end{eqnarray*} By Lemma \ref{doug gen unb},
		it follows that $\mathcal{R}(K)\subset\mathcal{R}(C_\psi^*)$.\\ 	$ii)\Rightarrow iii)$ By Lemma \ref{doug gen unb} there exists a bounded operator $M:\J\to L^2(X,\mu)$ such that $K=C_\psi^*M$.\\$iii)\Rightarrow i)$ $\psi$ is a continuous $K$-frame for $\H$ since	\begin{eqnarray*} \|K^*h\|_\J^2&=&\|M^*C_\psi h\|_\J^2 \leq \|M^*\|_{L^2,\J}^2\|C_\psi h\|_\J^2\\&=&\|M^*\|_{L^2,\J}^2\int_X
			|\ip{h}{\psi_x}_\H|^2 d\mu(x)\leq  \beta\|M^*\|_{L^2,\J}^2\|h\|_\H^2	\end{eqnarray*} by the boundedness of $C_\psi $.	\\$i)\Leftrightarrow iv)$ See \cite[Lemma 2.4]{AROR} with $\Lambda_x=\ip{f}{\psi_x}$ for every $f\in\H$, with $x\in X$.\\
			$i)\Rightarrow v)$ 	 The operator  $S_\psi $  is  positive, bounded and everywhere  defined in $\H$ because, by definition of continuous  $K$-frame for $\H$, there exists $\beta>0$ such that  $$0\leq\ip{S_\psi f}{f}_\H=\int_X|\ip{f}{\psi_x}_\H|^2d\mu(x)\leq\beta\|f\|_\H^2, \, \forall f\in\H.$$ 
			Hence $S_\psi =S_\psi^{1/2} S_\psi^{1/2}$, with $S_\psi^{1/2}$  positive self-adjoint operator and, by hypothesis, there exists $\alpha>0$ such that $$\alpha \|K^*f\|_\H^2\leq\left \|S_\psi^{1/2} f\right \|_\H^2,\qquad\forall f\in \H.$$
			By Lemma \ref{doug gen unb}, there exists $U\in\B(\J,\H)$ such that $K=\left(S_\psi^{1/2}\right) U.$\\
			$v)\Rightarrow i)$ By hypothesis  there exists $U\in\B(\J,\H)$ such that $
			K^*=
			\left(\left(S_\psi^{1/2}\right)U\right)^*=
			U^*S_\psi^{1/2}$, then, for every $f\in \H$
			$$		\|K^*f\|_\J^2=\left \|U^*S_\psi^{1/2}f \right \|_\J^2\leq\left \|U^*\right\|_{\H,\J}^2\left \|S_\psi^{1/2}f\right \|_\H^2\leq\left \|U^*\right\|_{\H,\J}^2\left\|S_\psi^{1/2}\right\|_{\H,\H}^2\left \|f\right \|_\H^2,	$$ hence $\psi$ is a continuous  $K$-frame for $\H$.\end{proof}

	\berem Nothing guarantees the closedness of $\mathcal{R}(C_\psi^* )$, then by Theorem \ref{thm: decompos op} $iii)$  it follows that a continuous $K$-frame is not automatically a continuous frame for the subspace  $\overline{\text{span}}{\{\psi_x\}}$, the closed linear
	span of $\{\psi_x\}$, which is in turn a Hilbert space (see \cite[Corollary 5.5.2]{ole} in the discrete case) 
	\enrem

\berem\label{rem: S everyw def} As usual, the frame operator $S_\psi$ of  a continuous $K$-frame for $\H$, with $K\in\B(\J,\H)$, is a linear positive bounded operator in $\H$, indeed $S_\psi =C_\psi^*C_\psi $ with $C_\psi \in \B(\H,L^2(X,\mu))$, however, it is not invertible in general. Nevertheless, if we strenghten the hypotheses on $K$ and $X$, $S_\psi $ can be invertible on its range (see 
\cite[p. 1245]{xzg} for the discrete case). \enrem
\begin{prop} Let $(X,\mu)$ be a $\sigma$-finite measure space,  $\psi:x\in X\to\psi_x\in\H$  a continuous $K$-frame for $\H$ with $K\in\B(\J,\H)$ having closed range. Then  $S_\psi $ is linear, bounded, self-adjoint, positive and invertible on $\mathcal{R}(K)$.
\end{prop}

\subsection{ Continuous atomic systems for unbounded operators $A$ and continuous $A$-frames}\label{subsec: cont atomic sys for unbounded op} The results of Section \ref{sec: cont K-frames} can be used to generalize continuous frames for bounded operators to the case of an unbounded closed and densely defined operator $A:\D(A)\to \H$ viewing it as a bounded operator between two different Hilbert spaces, more precisely, from  the Hilbert space $\H_A=\D(A)[\|\cdot\|_A]$ (where $\|\cdot\|_A$ is the graph norm induced by the graph inner product $\ip{\cdot}{\cdot}_A$) into $\H$. \\

In order to simplify notations, we come back to denote again by $\ip{\cdot}{\cdot}$ and $\|\cdot\|$ the inner product and the norm of $\H$, respectively.\\

We will indicate by  $A^\sharp:\H\to \H_A$  the adjoint of the bounded operator $A:\H_A\to\H$. With this convention, if $A\in\B(\H_A,\H)$, a function $\psi:x\in X\to \psi_x\in\H$  such that
	   for all $f\in\H$, the map $x \to \ip{f}{\psi_x}$ is a measurable function on $X$ is said to be
\begin{enumerate}
	\item[i)] a {\it continuous atomic system} for $A$ if $\psi$ is a  Bessel function and there exists $\gamma>0$  such that for all $f\in\D(A)$ there exists $a_f\in L^2(X,\mu)$, with
	$\|a_f\|_2=\left( \int_X |a_f(x)|^2d\mu(x)\right) ^{1/2}\leq \gamma\|f\|_A$ and for every $g\in\H$
	$$\ip{Af}{g}=\int_X a_f(x)\ip{\psi_x}{g} d\mu(x);$$
	\item[ii)] a {\it continuous $A$-frame} if there exist $\alpha, \beta>0$  such that for every $h\in \H$
	$$
	\alpha \|A^\sharp h\|_A^2 \leq \int_X
	|\ip{h}{\psi_x}|^2 d\mu(x) \leq \beta \|h\|^2.$$

\end{enumerate}
 Theorem \ref{thm: char cont atomic system} and Theorem \ref{thm: decompos op} can be summarized and rewritten as follows.

\begin{cor}
	\label{cor_A-frame_final}
	Let $\psi:x\in X\to \psi_x\in\H$ and suppose that for all $h\in\H$, the map $x \to \ip{h}{\psi_x}$ is a measurable function on $X$. Let $A$ be a closed densely defined operator on $\H$. Then the following are equivalent.
	\begin{itemize}
		\item[i)] $\psi$ is a continuous atomic system for $A$;
		\item[ii)] $\psi$ is a continuous $A$-frame;
		\item[iii)] $\psi$ is a Bessel function and there exists $\phi$ a  Bessel function  of $\H_A$  such that
		\begin{equation*}\label{exp_A}
				\ip{Af}{h}=\int_X \ip{f}{\phi_x}_A\ip{\psi_x}{h} d\mu(x),\qquad \forall f\in \D(A), \forall h\in\H;\end{equation*}
		\item[iv)] $C_\psi^*$  is bounded and $\mathcal{R}(A)\subset\mathcal{R}(C_\psi^*)$;
		\item[v)] $C_\psi^*$  is bounded and there exists $M\in \B(\H_A,L^2(X,\mu))$ such that $A=C_\psi^*M$.
		\item[vi)] $S_\psi =C_\psi^*C_\psi \geq\alpha AA^\sharp$ on $\H$, for some $\alpha>0$  and $\psi$ is a  Bessel function for $\H$;	\item[vii)]$A=\left(S_\psi^{1/2}\right)U$, for some $U\in\B(\H_A,\H)$. 
	\end{itemize}
\end{cor}

Note also that if $A\in \B(\H)$, then the graph norm of $A$ is defined on $\H$ and it is equivalent to $\|\cdot\|$, thus our notion of continuous $A$-frame  reduces to that of literature (see e.g. \cite{AROR}).

\section*{Acknowledgements} This work has been supported by the
Gruppo Nazionale per l'Analisi Matematica, la Probabilit\`{a} e le
loro Applicazioni (GNAMPA) of the Istituto Nazionale di Alta
Matematica (INdAM).

\vspace*{0.5cm}

\bibliographystyle{amsplain}

\end{document}